\documentclass[preprint]{elsarticle}
\usepackage{amsfonts,amssymb,amsmath}
\usepackage{algorithm}
\usepackage{algorithmic}

\newenvironment{proof}{\medskip                    
\noindent{\scshape Proof:}}{\quad $\square$
\medskip}  

\usepackage{hyperref}

\newtheorem{theorem}{Theorem}[section]
\newtheorem{lemma}[theorem]{Lemma}
\newtheorem{proposition}[theorem]{Proposition}
\newtheorem{corollary}[theorem]{Corollary}
\newtheorem{example}[theorem]{Example}
\newtheorem{remark}[theorem]{Remark}
\newtheorem{definition}[theorem]{Definition}

\newfont{\bb}{msbm10}

\def\b1{{\bf 1}}

\newcommand {\beq}{\begin{equation}}
\newcommand {\eeq} {\end{equation}}
\newcommand {\cP} {{\cal P}}

\newcommand {\R} {{\mathbb R}}

\def\C{{\rm C\kern-.48em\vrule width.06em height.6em depth-.02em
                 \kern.48em}}

\def\bce{\begin{center}}
\def\ece{\end{center}}

\if{
\documentclass[10pt]{article}
\usepackage{amsfonts,amssymb,amsmath}

\newenvironment{proof}{\medskip                    
\noindent{\scshape Proof:}}{\quad $\square$
\medskip}  

\usepackage{hyperref}

\newtheorem{theorem}{Theorem}[section]
\newtheorem{lemma}[theorem]{Lemma}
\newtheorem{proposition}[theorem]{Proposition}

\newtheorem{example}[theorem]{Example}
\newtheorem{remark}[theorem]{Remark}
\newtheorem{definition}[theorem]{Definition}

\usepackage{showkeys}
\newfont{\bb}{msbm10}

\defdiag{ {\rm diag}}

\def\b1{{\bf 1}}

\newcommand {\beq}{\begin{equation}}
\newcommand {\eeq} {\end{equation}}
\newcommand {\cP} {{\cal P}}

\newcommand {+} {\cP_{m}}
\newcommand {\R} {{\mathbb R}}

\def\C{{\rm C\kern-.48em\vrule width.06em height.6em depth-.02em
                 \kern.48em}}

\def\bce{\begin{center}}
\def\ece{\end{center}}

}\fi

\begin{document} 

\title{Bounding the row sum arithmetic mean by Perron roots of row-permuted matrices}

\author[rvt1]{Gernot Michael Engel}
\ead{engel@transversalnetworks.net}

\author[rvt2]{Serge{\u\i} Sergeev\corref{cor}}
\ead{s.sergeev@bham.ac.uk}

\address[rvt1]{Transversal Networks Corp., 2753 Marshall Parkway, Madison, WI 53713, USA}

\address[rvt2]{University of Birmingham, School of Mathematics, Edgbaston B15 2TT, UK}

\cortext[cor]{Corresponding author}



\begin{abstract}
\par

$\R_+^{n\times n}$ denotes the set of $n\times n$ non-negative matrices. For $A\in\R_+^{n\times n}$ let $\Omega(A)$ be the set of all matrices that can be formed by permuting the elements within each row of $A$. Formally: 
$$\Omega(A)=\{B\in\R_+^{n\times n}: \forall i\;\exists\text{ a permutation }\phi_i\;  \text{s.t.}\ b_{i,j}=a_{i,\phi_i(j)}\;\forall j\}.$$ 
For $B\in\Omega(A)$ let $\rho(B)$ denote the spectral radius or largest non negative eigenvalue of $B$. We show that the arithmetic mean of the row sums of $A$ is bounded by the maximum and minimum spectral radius of the matrices in $\Omega(A)$ Formally, we are showing that
$$\min_{B\in\Omega(A)}\rho(B)\leq \frac{1}{n}\sum_{i=1}^n\sum_{j=1}^n a_{i,j}\leq \max_{B\in\Omega(A)}\rho(B).$$
For positive $A$ we also obtain necessary and sufficient conditions for one of these inequalities (or, equivalently, both of them) to become an equality. We also give criteria which an irreducible matrix $C$ should satisfy to have $\rho(C)=\min_{B\in\Omega(A)} \rho(B)$ or $\rho(C)=\max_{B\in\Omega(A)} \rho(B)$. These criteria are used to derive algorithms for finding such $C$ when all the entries of $A$ are positive .
\end{abstract}


\maketitle
\section{Introduction}
\label{s:intro}

 In what follows, $\R_+^n$ denotes the set of non-negative vectors with length $n$ and $\R_+^{n\times n}$ denotes the set of non-negative $n\times n$ matrices. For $x\in\R_+^n$ or, respectively, $A\in\R_+^{n\times n}$ we write $x>0$ or, respectively, $A>0$, if all entries of vector $x$ or matrix $A$ are positive. We will work with the following matrix set, which can be defined for any matrix 
 $A$.
\begin{definition}
  For $A\in\R^{n\times n}_+$, the matrix set $\Omega(A)$ consists of the {\em row-permuted} matrices, whose entries in each row are a permutation of entries in the corresponding row of $A$. Formally: 
	\begin{equation}
 \Omega(A)=\{B\in\R_+^{n\times n}: \forall i\;\exists\text{a permutation }\phi_i\;  \text{s.t.}\ b_{i,j}=a_{i,\phi_i(j)}\;\forall j\}.	\end{equation}
\end{definition}

We will use the following standard notation for the Perron roots of matrices.
\begin{definition}	
The Perron root (i.e. the largest non negative eigenvalue, or spectral radius) of a matrix $B\in\Omega(A)$ will be denoted by $\rho(B)$. 
\end{definition}

For $A\in\R_+^{n\times n}$ the following row sum inequality 
 $$ \min_{i=1}^n \sum_{j=1}^n a_{i,j} \leq \rho(A) \leq \max_{i=1}^n \sum_{j=1}^n a_{i,j} $$
 was first observed by Frobenius.
 The geometric means of the row sums as bounds for $\rho(A)$ were explored by Al'pin~\cite{Al} and Elsner and van Driessche~\cite{ED}, and further generalised by Engel et al.~\cite{ESS}.
 In this paper we are interested in establishing a different connection between Perron roots and row sums. Namely, we show that the  arithmetic mean of the row sums satisfies
	\begin{equation}
	\min_{B\in\Omega(A)}\rho(B)\leq \frac{1}{n}\sum_{i=1}^n\sum_{j=1}^n a_{i,j}\leq \max_{B\in\Omega(A)}\rho(B).
	\end{equation}
For $A>0$ we  obtain necessary and sufficient conditions for any of these inequalities to turn into equalities. For $A\in\R_+^{n\times n}$ we also give necessary and sufficient criteria for an irreducible matrix $C\in\Omega(A)$ to have $\rho(C)=\min_{B\in\Omega(A)}\rho(B)$ or $\rho(C)=\max_{B\in\Omega(A)}\rho(B)$.

To obtain these results we make use, in particular, of the following well-known facts. These facts, which we are going to use throughout the paper, are closely related to the famous Collatz-Wielandt inequality and are summarized in the following proposition:

\begin{proposition}[e.g., \cite{BP}, Theorem 1.11]
\label{BP28}
For $ A\in\R_+^{n\times n}$ and constants $\alpha>0,\beta>0$ and nonzero vector
 $x\in\R_{+}^n$ we have:
\begin{itemize}
   \item[{\rm (i)}] $\alpha x\leq Ax$\ implies $\alpha\leq\rho(A)$,
   \item[{\rm (ii)}] $Ax\leq \beta x$ with $x>0$ implies $\rho(A)\leq \beta$.
\end{itemize}
In addition, if $A$ is irreducible then the following implications hold:
\begin{itemize}
\item[{\rm (iii)}] if $\alpha x \leq A x$ and $\exists i$ such that $\alpha x_i< \sum_{j=1}^n a_{i,j}x_j$ then $\rho(A)>\alpha$, 
\item[{\rm (iv)}] if $A x\leq \beta x$ and $i$ such that $\sum_{j=1}^n a_{i,j}x_j<\beta x_i$ then $\rho(A)<\beta$.
\end{itemize}
\end{proposition}

The next result, which we will use to derive the criteria for $\rho(C)=\max_{B\in\Omega(A)} \rho(B)$ and 
$\rho(C)=\min_{B\in\Omega(A)} \rho(B)$, is known as the {\em rearrangement inequality.}

\begin{proposition}[e.g., \cite{HLP}, page 261]
\label{p:rearrange}
Let $x,y\in\R_+^n$ be such that $x_1\leq x_2\leq \ldots\leq x_n$ and $y_1\leq y_2\leq\ldots\leq y_n$, and let 
$\phi\colon\{1,\ldots,n\}\to\{1,\ldots,n\}$ be an arbitrary permutation. Then the following inequalities hold:
$$
\sum_{i=1}^n x_iy_{n+1-i}\leq \sum_{i=1}^n x_iy_{\phi(i)}\leq \sum_{i=1}^n x_iy_i.
$$
\end{proposition}


\section{Preliminary lemmas}
 \label{s:lemmas}

The following lemma, related to the rearrangement inequality, establishes that the maximum Perron root is achieved on a matrix with all positive entries for which the correlation between the order of the components of its Perron eigenvector and each of its row vectors is maximized. The minimum Perron root is achieved when the correlation between the order of the components of its Perron eigenvector and each of its row vectors is minimized.

\begin{lemma}
\label{primary tool}
Let $A\in\R_+^{n\times n}$ be irreducible. Then the following implications hold:
\begin{itemize}
\item[{\rm (i)}] 
if $\rho(A)=\max_{B\in\Omega(A)}\rho(B)$ and $x$ is a Perron eigenvector of $A$
then for $1\leq i,j,k\leq n$:
$$x_k < x_j \text{ implies }a_{i,k}\leq a_{i,j},$$
  \item[{\rm (ii)}]
	if $\rho(A)=\min_{B\in\Omega(A)}\rho(B)$ and $x$ is a Perron eigenvector of $A$
then for $1\leq i,j,k\leq n$:

$$x_k < x_j \text{ implies }a_{i,k}\geq a_{i,j}.$$
\end{itemize}
\end{lemma}
\begin{proof} (i): Assume that $A\in\R_+^{n\times n}$ is irreducible and  $\rho(A)=\max_{B\in\Omega(A)}\rho(B)$ . Then $\exists x>0$ such that $Ax=\rho(A)x$. By contradiction, assume that there exist $i,j,k$ such that $x_j>x_k$ but $a_{i,j}<a_{i,k}$. Let $B$ be the matrix formed by swapping the two entries $a_{i,j}$ and $a_{i,k}$ so that $b_{i,j}=a_{i,k}$ and $b_{i,k}=a_{i,j},$ with all other entries of $B$ equal to the entries of $A$. Then
 $B$ is in $\Omega(A)$. We have  $\rho(A)x_i= \sum_{j=1}^n a_{i,j}x_j<\sum_{j=1}^n b_{i,j}x_j$  and $\sum_{j=1}^n a_{s,j}x_j=\sum_{j=1}^n b_{s,j}x_j$ for $s\neq i$. Thus by Proposition \ref{BP28} part (iii), $\rho(B)>\rho(A)$. Since this contradicts that  $\rho(A)=\max_{B\in\Omega(A)}\rho(B)$ it follows that  for $1\leq i,j,k \leq n\ x_k < x_j$ implies $a_{i,k}\leq a_{i,j}$ and (i) is established.

(ii): The proof is similar to the previous part, with the difference that here we assume that $\rho(A)=\min_{B\in\Omega(A)}\rho(B)$. 
Upon assuming by contradiction that there exist $i,j,k$ such that $x_j>x_k$ but $a_{i,j}>a_{i,k}$ we define matrix $B$ by swapping the entries $a_{i,j}$ and $a_{i,k}$ so that $b_{i,j}=a_{i,k}$ and $b_{i,k}=a_{i,j},$ with all other entries of $B$ equal to the entries of $A$. Observing that $\rho(A)x_i= \sum_{j=1}^n a_{i,j}x_j>\sum_{j=1}^n b_{i,j}x_j$  and  
$\sum_{j=1}^n a_{s,j}x_j=\sum_{j=1}^n b_{s,j}x_j$ for $s\neq i$, we use  Proposition \ref{BP28} part (iv) to obtain $\rho(B)<\rho(A)$,
a contradiction establishing part (ii). \end{proof}


Proof of the next lemma follows the reasoning used in the proof of Tchebychef's inequality \cite{HLP} page 43. 
\begin{lemma}
\label{Tchebychef}
Let $A\in\R_+^{n\times n}$ have a Perron eigenvector $x\in\R_+^n$  satisfying $\sum_{i=1}^n x_i=1$. Then the 
following properties hold:
\begin{itemize}
\item[{\rm (i)}] if $\forall i,j,k\ x_k < x_j$ implies $a_{i,k}\leq a_{i,j}$ then 
$$\forall i\ \frac{\sum_{j=1}^n a_{i,j}}{n}\leq \sum_{j=1}^n a_{i,j}x_j=\rho(A)x_i,$$

\item[{\rm (ii)}] if $\forall i,j,k\ x_k < x_j$ implies $a_{i,k}\geq a_{i,j}$ then 
$$\forall i\ \frac{\sum_{j=1}^n a_{i,j}}{n}\geq \sum_{j=1}^n a_{i,j}x_j=\rho(A)x_i,$$

\item[{\rm (iii)}] if $\forall i,j,k\ x_k < x_j$ implies $a_{i,k}\leq a_{i,j}$  or $\forall i,j,k\ x_k < x_j$ implies $a_{i,k}\geq a_{i,j}$,  then  the following are equivalent:\\
\begin{itemize}
\item[{\rm (a)}] $\frac{\sum_{j=1}^n a_{i,j}}{n} =\sum_{j=1}^n a_{i,j}x_j=\rho(A)x_i$ for all $i$;\\
\item[{\rm (b)}] either $x_i =\frac{1}{n}$ for all $i$,  or for each $i$ there is $c_i$ such that $c_i=a_{i,j}$ for all $j$.
\end{itemize}
\end{itemize}
\end{lemma}

\begin{proof} (i): 
The property that  $\forall 1\leq i,j,k \leq n$ $x_k< x_j$ implies $a_{i,k}\leq a_{i,j}$ is equivalent to
$(a_{i,j}-a_{i,k})(x_j-x_k)\geq 0$.  From this we obtain
\begin{equation*}
\begin{split}
\forall i\ 2n x_i\rho(A)&=\sum_{j=1}^n\sum_{k=1}^n (a_{i,j}x_j+a_{i,k}x_k)\geq \sum_{j=1}^n\sum_{k=1}^n (a_{i,j}x_k+a_{i,k}x_j)\\
&\geq 2\left(\sum_{j=1}^n a_{i,j}\right)\left(\sum_{j=1}^n x_j\right)=2\sum_{j=1}^n a_{i,j}. 
\end{split}
\end{equation*}
This implies
\begin{equation*}
\forall i\ \frac{\sum_{i=1}^n a_{i,j}}{n}\leq \sum_{i=1}^n a_{i,j} x_j=\rho(A)x_i
\end{equation*}
establishing part (i).

\noindent (ii):  The proof of this part is similar to the proof of part (i). Here we first observe that the property that 
 $\forall 1\leq i,j,k \leq n$ $x_k< x_j$ implies $a_{i,k}\geq a_{i,j}$ is equivalent to
$(a_{i,j}-a_{i,k})(x_j-x_k)\leq 0$. Using this inequality in the same way as in the proof of part (i) the opposite inequality is used, we obtain 
 \begin{equation*}
 \forall i\ \frac{\sum_{i=1}^n a_{i,j}}{n}\geq \sum_{i=1}^n a_{i,j}x_j=\rho(A).
 \end{equation*} 
 establishing part (ii).


\noindent (iii):
To establish (3) (a) implies (b) assume $\frac{\sum_{i=1}^n a_{i,j}}{n} =\sum_{i=1}^n a_{i,j}x_j=\rho(A)x_i$ and that either $\forall 1\leq i,j,k \leq n\ x_k < x_j$ implies $a_{i,k}\leq a_{i,j}$  or  $\forall 1\leq i,j,k \leq n\ x_k < x_j$ implies $a_{i,k}\geq a_{i,j}.$

\noindent In the first case for any $i,j,k$ we have that $(a_{i,j}-a_{i,k})(x_j-x_k)\geq 0$ and in the second case we have that $(a_{i,j}-a_{i,k})(x_j-x_k)\leq 0$. In the first case, if there exists $i$ such that $(a_{i,j}-a_{i,k})(x_j-x_k)> 0$ for some $j$ and $k$ then 
$\frac{\sum_{i=1}^n a_{i,j}}{n}<\sum_{i=1}^n a_{i,j}x_j=\rho(A)x_i.$ Similarly in the second case if there exists $i$ such that $(a_{i,j}-a_{i,k})(x_j-x_k)< 0$, then  $\frac{\sum_{i=1}^n a_{i,j}}{n}>\sum_{i=1}^n a_{i,j}x_j=\rho(A)x_i.$ Since none of these strict inequalities holds, we have 
\begin{equation}
\label{e:aijxequal0}
\forall i,j,k\ (a_{i,j}-a_{i,k})(x_j-x_k)= 0.
\end{equation}

For any $i=1,\ldots,n$ let $t(i)$ and $d(i)$ be defined (non-uniquely) by 
\begin{equation}
\label{e:tidi}
a_{i,t(i)}=\min_j a_{i,j},\quad a_{i,d(i)}=\max_j a_{i,j}
\end{equation}
 and suppose that $x_i=x_k$ does not hold for all $i\neq k$. Our aim is to show that then the coefficients in every row of $A$ are equal to each other. Since either $\forall i,j,k$ $x_k<x_j$ implies $a_{i,k}\leq a_{i,j}$ or $\forall i,j,k$ $x_k<x_j$ implies $a_{i,k}\geq a_{i,j}$, we can let $t(i)$ and $d(i)$ be defined in such a way that not only equalities \eqref{e:tidi} hold but also in the first case $x_{t(i)}=\min_j x_j$ and $x_{d(i)}=\max_j x_j$ and in the second case $x_{t(i)}=\max_j x_j$ and $x_{d(i)}=\min_j x_j$. In both cases \eqref{e:aijxequal0} entails that
 $(a_{i,t(i)}-a_{i,d(i)})(x_{t(i)}-x_{d(i)})=0$ and hence $a_{i,t(i)}=a_{i,d(i)}$. By \eqref{e:tidi} we obtain that all all entries in the $i$th row of $A$ are equal to each other, establishing the implication (a)$\Rightarrow$(b). 


To prove that (b) implies (a), first observe that obviously if $x_j =\frac{1}{n}$  then $\frac{\sum_{i=1}^n a_{i,j}}{n} =\sum_{i=1}^n a_{i,j}x_j=\rho(A)x_i.$
If instead for each $i$ there is $c_i$ such that $a_{i,j}=c_i$ for all $j$, then the unique Perron eigenvector $x$ with 
$\sum_{j=1}^n x_j=1$ has coordinates $x_i=c_i/\sum_{j=1} c_{j}$ for all $i$ and the Perron root is 
$\rho(A)=\sum_{j=1}^n c_j$. Indeed, we have 
$$
\sum_{j=1}^n a_{i,j}x_j=c_i\sum_{i=1}^n x_j=c_i=\sum_{j=1}^n c_j\cdot\frac{c_i}{\sum_{j=1}^n c_j}=\rho(A)x_i.
$$
In this case 
$\frac{\sum_{i=1}^n a_{i,j}}{n} = c_i=\sum_{j=1}^n a_{i,j}x_j,$
establishing (a).
\end{proof}

\section{Main results}
\label{s:main}

We begin this section by establishing the inequality between the arithmetic mean of the rows and the largest and smallest Perron roots of matrices in $\Omega(A)$.

\begin{theorem}
\label{main}
For any $A\in\R_+^{n\times n}$   
\begin{equation}
\label{e:main}
   \min_{B\in\Omega(A)}\rho(B)\leq\frac{1}{n} \sum_{i=1}^n \sum_{j=1}^n a_{i,j}\leq \max_{B\in\Omega(A)}\rho(B).
\end{equation}
\end{theorem}
\begin{proof}
We first assume that $A>0$ and establish $\frac{1}{n} \sum_{i=1}^n \sum_{j=1}^n a_{i,j}\leq \max_{B\in\Omega(A)}\rho(B)$
for such $A$. 
Select $C\in\Omega(A)$ such that $\rho(C)=\max_{B\in\Omega(A)}\rho(B)$. Let $x>0$ be a Perron eigenvector of $C$ such that $\sum_{i=1}^n x_i=1$. By Lemma \ref{primary tool} part (i) we have that for $1\leq i,j,k \leq n\ x_k < x_j$ implies $c_{i,k}\leq c_{i,j}$. 
Then by Lemma \ref{Tchebychef} part (i) we have  $\forall i\  \frac{1}{n x_i} \sum_{j=1}^n c_{i,j}\leq \sum_{j=1}^n c_{i,j} \frac{x_j}{x_i}=  \max_{B\in\Omega(A)}\rho(B)$ and since 
$\sum_{j=1}^n c_{i,j}=\sum_{j=1}^n a_{ij}$ for all $i$, we obtain  
$$\frac{1}{n} \sum_{i=1}^n \sum_{j=1}^n a_{i,j}\leq \sum_{i=1}^n\max_{B\in\Omega(A)}\rho(B) x_i=\max_{B\in\Omega(A)}\rho(B).$$
\\
\\
Still assuming $A>0$, we can establish $\frac{1}{n} \sum_{i=1}^n \sum_{j=1}^n a_{i,j}\geq \min_{B\in\Omega(A)}\rho(B)$ in a similar way. For this we select $C\in\Omega(A)$ such that $\rho(C)=\min_{B\in\Omega(A)}\rho(B)$ and let $x>0$ be a Perron eigenvector of $C$ such that $\sum_{i=1}^n x_i=1$. Combining Lemma \ref{primary tool} part (ii) with Lemma \ref{Tchebychef} part (ii) we obtain 
 $\forall i\  \frac{1}{n x_i} \sum_{j=1}^n c_{i,j}\geq \sum_{j=1}^n c_{i,j} \frac{x_j}{x_i}$ and hence 
 $$\frac{1}{n} \sum_{i=1}^n \sum_{j=1}^n a_{i,j}\geq \sum_{i=1}^n\min_{B\in\Omega(A)}\rho(B) x_i=\min_{B\in\Omega(A)}\rho(B).$$


Now for arbitrary $A\in\R_+^{n\times n}$ and $\epsilon>0$ we define $A^{\epsilon}=(a^{\epsilon}_{i,j})=(a_{i,j}+\epsilon)$. Then since $0<A^{\epsilon}$ we have that  $\min_{B\in\Omega(A^{\epsilon})}\rho(B)\leq\frac{1}{n}\sum_{i=1}^n \sum_{j=1}^n (a_{i,j}+\epsilon)\leq \max_{B\in\Omega(A^{\epsilon})}\rho(B).$ 
Thus by continuity of the Perron root  and letting $\epsilon$ go to zero we obtain the desired inequality for $A$. 
\end{proof}

We now establish the conditions when any of the inequalities in Theorem \ref{main} becomes an equality.
\begin{theorem}
For $0<A\in\R_+^{n\times n}$ the following are equivalent:
\begin{itemize}
	\item[{\rm (i)}]  $\frac{1}{n} \sum_{i=1}^n \sum_{j=1}^n a_{i,j}= \max_{B\in\Omega(A)}\rho(B)$.
	\item[{\rm (ii)}] Either the flat vector $x=(x_i)$ where $\forall i\ x_i=1$ is a Perron eigenvector of $A$ or 
	  there exists a non singular diagonal matrix $\exists\ D\geq 0$ such that $DA$ is a flat matrix (i.e. $\forall i,j\ d_i a_{i,j}=1$ ).
	\item[{\rm (iii)}] $\frac{1}{n} \sum_{i=1}^n \sum_{j=1}^n a_{i,j}= \min_{B\in\Omega(A)}\rho(B)$.
\end{itemize}
\end{theorem}
\begin{proof}
We first establish (i)$\Rightarrow$(ii). By (i), $\frac{1}{n} \sum_{i=1}^n \sum_{j=1}^n a_{i,j}= \max_{B\in\Omega(A)}\rho(B).$
 Since the set $\Omega(A)$ is finite, there exist $C\in\Omega(A)$ and $y\in\R_+^{n\times n}$ with $\sum_{i=1}^n y_i=1$  such that $Cy=(\max_{B\in\Omega(A)}\rho(B))y$. Thus 
 $$\sum_{i=1}^n\sum_{j=1}^n c_{i,j}y_j= \left(\max_{B\in\Omega(A)}\rho(B)\right)\left(\sum_{i=1}^n y_i\right)= \max_{B\in\Omega(A)}\rho(B)= \sum_{i=1}^n \sum_{j=1}^n a_{i,j}\frac{1}{n}.$$ 
 By Lemma \ref{primary tool} part (i) $\forall i,j,k\colon  1\leq i,j,k \leq n$ we have that $y_k<y_j$ implies $c_{i,k}\leq c_{i,j}$. 
Then by Lemma \ref{Tchebychef} part (i) 
$$\forall i\  \frac{1}{n}\sum_{j=1}^n c_{i,j}\leq \sum_{j=1}^n c_{i,j}y_j=\max_{B\in\Omega(A)}\rho(B)y_i.$$ 
Since $\forall i\  \frac{1}{n}\sum_{j=1}^n a_{i,j}= \frac{1}{n}\sum_{j=1}^n c_{i,j}$ we can rewrite this as
$$\forall i\  \frac{1}{n}\sum_{j=1}^n a_{i,j}\leq \max_{B\in\Omega(A)} \rho(B)y_i.$$ 
As by (i) we have  $$\frac{1}{n} \sum_{i=1}^n \sum_{j=1}^n a_{i,j}=\max_{B\in\Omega(A)}\rho(B)=\max_{B\in\Omega(A)}\rho(B)\sum_{i=1}^n y_i,$$
if there exists $i$ such that $\frac{1}{n} \sum_{j=1}^n a_{i,j} <\max_{B\in\Omega(A)}\rho(B)y_i$ then there would have to exist $k$ such that $\frac{1}{n} \sum_{j=1}^n a_{k,j} >\max_{B\in\Omega(A)}\rho(B)y_k$, which is a contradiction, hence
$$\forall i \max_{B\in\Omega(A)}\rho(B)y_i= \frac{1}{n} \sum_{j=1}^n a_{i,j}=\sum_{j=1}^n c_{i,j}y_j= \frac{1}{n} \sum_{j=1}^n c_{i,j}.$$
Applying Lemma \ref{Tchebychef} part (iii), we obtain that either $\forall i\ y_i=1/n$ or $\forall i,j,k\ c_{i,j}=c_{i,k}$. If $\forall i\ y_i=1/n$ then $\forall i\ \frac{1}{n} \sum_{j=1}^n a_{i,j}=\max_{B\in\Omega(A)}\rho(B)\frac{1}{n}$, from which it follows that  flat vector $x=(x_i)$ where $\forall i\ x_i=1$ is a Perron eigenvector of $A.$ If $\forall i,j,k\ c_{i,j}=c_{i,k}$ then $\forall i,j,k\ a_{i,j}=c_{i,j}=c_{i,k}=a_{i,k}.$  Let $D$ be the diagonal matrix where $\forall i\ d_{i,i}=\frac{1}{a_{i,i}}$ and the rest of the entries of $D$ are $0$. Thus $DA$ is the flat matrix such that $\forall i,j\ d_i a_{i,j}=1.$

We now show (ii)$\Rightarrow$(i),(iii). Assume first that the flat vector $x=(x_i)$ where $\forall i\ x_i=1$ is a Perron eigenvector of $A$. This is equivalent to all row sums of $A$ being equal to each other. If this property holds for $A$ then it also holds for all $B\in\Omega(A)$, so the flat vector is a Perron eigenvector of any such $B$ with the same Perron root (equal to any of the row sums). Thus we have both (i) and (iii), i.e.,
\begin{equation}
\label{e:maxmineq}
\max_{B\in\Omega(A)}\rho(B)=\min_{B\in\Omega(A)}\rho(B)=\frac{1}{n} \sum_{i=1}^n \sum_{j=1}^n a_{i,j}.
\end{equation}
Now assume that there exists a non-singular diagonal matrix $D\geq 0$  such that $DA$ is a flat matrix. In this case the entries in each row of $A$ are equal to each other, implying that $\Omega(A)=\{A\}$. As the left hand side and the right hand side of~\eqref{e:main} are equal to each other, we obtain~\eqref{e:maxmineq}.

Finally, the proof of (iii)$\Rightarrow$ (ii) is similar to the proof of (i)$\Rightarrow$(ii) and will be described more briefly. By (ii), $\frac{1}{n} \sum_{i=1}^n \sum_{j=1}^n a_{i,j}= \min_{B\in\Omega(A)}\rho(B).$
 Since the set $\Omega(A)$ is finite, there exist $C\in\Omega(A)$ and $y\in\R_+^{n\times n}$ with $\sum_{i=1}^n y_i=1$  such that $Cy=(\min_{B\in\Omega(A)}\rho(B))y$. Thus 
 $$\sum_{i=1}^n\sum_{j=1}^n c_{i,j}y_j= \left(\min_{B\in\Omega(A)}\rho(B)\right)\left(\sum_{i=1}^n y_i\right)= \min_{B\in\Omega(A)}\rho(B)= \sum_{i=1}^n \sum_{j=1}^n a_{i,j}\frac{1}{n}.$$ 
 Next, combining Lemma~\ref{primary tool} part (ii) and Lemma~\ref{Tchebychef} part (ii) and using that for each $i$ the sum of the $i$th row of $A$ equals the sum of the $i$th row of $C$, we obtain 
 $$\forall i\  \frac{1}{n}\sum_{j=1}^n a_{i,j}\geq \min_{B\in\Omega(A)} \rho(B)y_i.$$  
 Using condition (iii), however, we see that the strict inequality cannot hold for any $i$ and therefore we have
 $$\forall i \min_{B\in\Omega(A)}\rho(B)y_i= \frac{1}{n} \sum_{j=1}^n a_{i,j}=\sum_{j=1}^n c_{i,j}y_j= \frac{1}{n} \sum_{j=1}^n c_{i,j}.$$
Condition (ii) then follows by applying Lemma~\ref{Tchebychef} part (iii) (see the end of the proof of (i)$\Rightarrow$(ii) written above.) \end{proof}  
 \if{
 By Lemma \ref{primary tool} part (i) $\forall i,j,k\colon  1\leq i,j,k \leq n$ we have that $y_k<y_j$ implies $c_{i,k}\leq c_{i,j}$. 
Then by Lemma \ref{Tchebychef} part (i) 
$$\forall i\  \frac{1}{n}\sum_{j=1}^n c_{i,j}\leq \sum_{j=1}^n c_{i,j}y_j=\max_{B\in\Omega(A)}\rho(B)y_i.$$ 
Since $\forall i\  \frac{1}{n}\sum_{j=1}^n a_{i,j}= \frac{1}{n}\sum_{j=1}^n c_{i,j}$ we can rewrite this as
$$\forall i\  \frac{1}{n}\sum_{j=1}^n a_{i,j}\leq \max_{B\in\Omega(A)} \rho(B)y_i.$$

Assume  $\frac{1}{n} \sum_{i=1}^n \sum_{j=1}^n a_{i,j}= \min_{B\in\Omega(A)}\rho(B)$. Since the set $\Omega(A)$ only has a finite number of members there exists $C\in\Omega(A)$ and $y\in\R_+^{n\times n}$ with $\sum_{i=1}^n y_i=1$  such that $Cy=(\min_{B\in\Omega(A)}\rho(B))y$. Thus $\sum_{i=1}^n\sum_{j=1}^n c_{i,j}y_j= (\min_{B\in\Omega(A)}\rho(B))(\sum_{i=1}^n y_i)= \min_{B\in\Omega(A)}\rho(B)= \sum_{i=1}^n \sum_{j=1}^n a_{i,j}\frac{1}{n}.$ By lemma \ref{primary tool} $\forall i,j,k,\  1\leq i,j,k \leq n\ \text{ we have that } y_k\leq y_j$ implies $c_{i,k}\geq c_{i,j}$. 
\newline
\par
Thus by lemma \ref{Tchebychef} $\forall i\  \frac{1}{n}\sum_{j=1}^n c_{i,j}\geq \sum_{j=1}^n c_{i,j}y_j=\min_{B\in\Omega(A)}\rho(B)y_i$. Since $\forall i\  \frac{1}{n}\sum_{j=1}^n a_{i,j}= \frac{1}{n}\sum_{j=1}^n c_{i,j}$ we have that $\forall i\  \frac{1}{n}\sum_{j=1}^n a_{i,j}\leq \min_{B\in\Omega(A)} \rho(B)y_i$. Since  $\frac{1}{n} \sum_{i=1}^n \sum_{j=1}^n a_{i,j}=\min_{B\in\Omega(A)}\rho(B)=\min_{B\in\Omega(A)}\rho(B)\sum_{i=1}^n y_i$ if there exists $i$ such that $\frac{1}{n} \sum_{j=1}^n a_{i,j} >\min_{B\in\Omega(A)}\rho(B)y_i$ then there would have to exist $k$ such that $\frac{1}{n} \sum_{j=1}^n a_{k,j} <\min_{B\in\Omega(A)}\rho(B)y_k$ which is a contradiction.
\\
\\
 Thus $\forall i \min_{B\in\Omega(A)}\rho(B)y_i= \frac{1}{n} \sum_{j=1}^n a_{i,j}=\sum_{j=1}^n c_{i,j}y_j= \frac{1}{n} \sum_{j=1}^n c_{i,j}.$ Thus by lemma \ref{Tchebychef} item (3) $\forall i\ y_i=1/n$ or $\forall i,j,k\ c_{i,j}=c_{i,k}$. If $\forall i\ y_i=1/n$ then $\forall i\ \frac{1}{n} \sum_{j=1}^n a_{i,j}=\min_{B\in\Omega(A)}\rho(B)\frac{1}{n}$ from which it follows that  flat vector $x=(x_i)$ where $\forall i\ x_i=1$ is a Perron eigenvector of $A.$ 
\\
\\
If $\forall i,j,k\ c_{i,j}=c_{i,k}$ then $\forall i,j,k\ a_{i,j}=c_{i,j}=c_{i,k}=a_{i,k}.$  Let $D$ be the diagonal matrix where $\forall i\ d_{i,i}=\frac{1}{a_{1,j}}$ then $DA$ is the flat matrix such that $\forall i,j\ d_i a_{i,j}=1.$ 
}\fi

\if{
\begin{remark}
\label{total ordering}
For any vector $v\in\R_+^n$ the coefficients of $v$ are partially ordered by $\leq$. In this paper this  ordering  is extended into a total ordering $<_\tau$ that is consistent with the usual partial ordering $\leq$. The rule used to obtain this extension is that when $v_i$ and $v_j$ have the same value that if $i<j$ then $v_i<_\tau v_j$ .
\end{remark}

\begin{lemma}
\label{total}
Select $0<A\in\R_+^{n\times n}$ such that $\rho(A)= \max_{B\in\Omega(A)}\rho(B)$ with eigenvector $x\in\R_+^n$ satisfying $Ax=\rho(A)x.$ Then there exists $B\in\Omega(A)$ such that $\rho(B)=\rho(A)$ and $\forall\ i,j,k\ 1\leq i,j,k\leq n\ x_k <_\tau x_j \text{ implies } b_{i,k}<_{\tau} b_{i,j}.$ 
\end{lemma}
\begin{proof}
Select $0<A\in\R_+^{n\times n}$ such that $\rho(A)= \max_{B\in\Omega(A)}\rho(B)$ with eigenvector $x\in\R_+^n$ satisfying $Ax=\rho(A)x.$ Then by lemma \ref{primary tool}  $\forall\ i,j,k\ 1\leq i,j,k\leq n\ x_k < x_j \text{ implies } a_{i,k}\leq a_{i,j}$. 
\\
\\ 
Define $B=(b_{i,j})$ where $b_{i,j} = a_{i,k}$ where  $x_k=x_j$ and $\forall x_s\leq x_k\ a_{i,s}\leq a_{i,k}$ and $\forall x_s\geq x_k\ a_{i,s}\geq a_{i,k}$ .  Then  $\forall\ i,j,k\ 1\leq i,j,k\leq n\ x_k <_{\tau} x_j \text{ implies } b_{i,k}<_{\tau} b_{i,j}$ and $\forall i\ x_i\rho(A)=\sum_{j=1}^n a_{i,j}x_j=\sum_{j=1}^n b_{i,j}x_j.$
\end{proof}
}\fi

\if{
\begin{lemma}
 Let $A\in\R_+^{n\times n}$ be  such that every row  of $A$ has at least one non zero entry.
\\
\\
 For $ C\in\Omega(A)$ with Perron vector $x\geq 0$ the condition
 $$\forall\ i,j,k\ 1\leq i,j,k\leq n: x_k < x_j\;\Rightarrow\;c_{i,k}\leq c_{i,j}$$
implies  $x>0.$
\end{lemma}
\begin{proof}
Assume $ C\in\Omega(A)$ with Perron vector $x\geq 0$ satisfies 
 $(\forall\ i,j,k\ 1\leq i,j,k\leq n: x_k < x_j\;\Rightarrow\;c_{i,k}\leq c_{i,j}$.
\\
\\
Since $x$ is a Perron vector of $C$ there exists  $i$ with  $x_i>0.$ Since every row of $A$ has a non zero entry there exists $U$ in $\Omega(A)$  such that each row and column of $U$ contains a off diagonal non zero entry. Select $j\neq i$. Then there exists a sequence $j=t_1,t_2,\cdots,t_k=i$ such that $u_{t_s,t_{s+1}}>0$ for $1\leq s <k.$ Since $x_i>0$ and $u_{t_{k-1},i}>0$ we have that for $x_{v}<x_i$ that if $c_{t_{k-1},i}=0$ that $c_{t_{k-1},v}=0$. Thus either $c_{t_{k-1},i}>0$ or there exists $s$ such that $x_i<x_s$ and $c_{t_{k-1},s}=u_{t_s,t_{s+1}}>0$. In either case $0<Cx|_{t_{k-1}}= \rho(C)x_{t_{k-1}}.$ Thus $x_{t_{k-1}}>0.$
\\
\\
By repeating this argument for $x_{k-1},\ldots,x_{t_2}.$ we show that for $2\leq s\leq k$ that $x_{t_s}>0$ implies $x_{t_{s-1}}=0.$ Thus $x_j>0.$ Since $j$ was selected arbitrarly we conclude that  $x>0.$ 
\end{proof}
}\fi

The following result applies the rearrangement inequality  (Proposition~\ref{p:rearrange}) to yield a sufficient condition  for establishing when $\rho(A)=\max_{B\in\Omega(A)}\rho(B)$ and $\rho(A)=\min_{B\in\Omega(A)}\rho(B)$.
\begin{theorem}
\label{sufficient}
Let $A\in\R_+^{n\times n}$  and $0\leq x\in\R_+^n$ be a Perron eigenvector of $A$. Then
\begin{equation}
\label{sufmax}
\left(\forall\ i,j,k\ 1\leq i,j,k\leq n: x_k < x_j\;\Rightarrow\; a_{i,k}\leq a_{i,j}\right)\Longrightarrow \rho(A)=\max_{B\in\Omega(A)}\rho(B) 
\end{equation}
\begin{equation}
\label{sufmin}
  \left(\forall\ i,j,k\ 1\leq i,j,k\leq n: x_k < x_j\;\Rightarrow\;a_{i,k}\geq a_{i,j}\right)\Longrightarrow\rho(A)=\min_{B\in\Omega(A)}\rho(B) 
\end{equation}
\end{theorem}
\begin{proof}
Consider the condition on the left hand side of~\eqref{sufmax}. Observe that we can assume without loss of generality that $a_{i,k}\leq a_{i,j}\Leftrightarrow a_{l,k}\leq a_{l,j}$ for any two rows $i$ and $l$ of $A$. Indeed, if $x_k<x_j$ then this is the case (by the condition), and if $x_k=x_j$ then the entries $a_{i,k}$ and $a_{i,j}$ or $a_{l,k}$ and $a_{l,j}$ can be swapped without changing $Ax$, so that the modified matrix belongs to $\Omega(A)$ and has the same Perron eigenvector $x$ and the same Perron root $\rho(A)$. 
Then we can also assume without loss of generality that simultaneously $x_1\leq x_2\leq\ldots \leq x_n$ and $a_{i,1}\leq a_{i,2}\leq\ldots\leq a_{i,n}$ for all $i$. If we consider any matrix $B\in\Omega(A)$, then the rearrangement inequality implies that 
$Bx\leq Ax=\rho(A)x$ and hence $\rho(B)\leq \rho(A)$.

Similarly, to prove the sufficiency of the condition on the right hand side of \eqref{min}, we can assume without loss of generality that $a_{i,k}\geq a_{i,j}\Leftrightarrow a_{l,k}\geq a_{l,j}$ for any two rows $i$ and $l$ of $A$. Indeed, if $x_k<x_j$ then this is the case (by the condition),  and if $x_k=x_j$ then the corresponding non-aligning entries in any row can be swapped to obtain the alignment. Then we can also assume without loss of generality that simultaneously $x_1\leq x_2\leq\ldots \leq x_n$ and $a_{i,1}\geq a_{i,2}\geq\ldots\geq a_{i,n}$ for all $i$. If we consider any matrix $B\in\Omega(A)$, then the rearrangement inequality implies that 
$Bx\geq Ax=\rho(A)x$ and hence $\rho(B)\geq \rho(A)$. 
\end{proof}

The following result applies Lemma~\ref{primary tool} to show that for  irreducible matrices conditions~\eqref{sufmax} and~\eqref{sufmin} of Theorem \ref{sufficient} are necessary and sufficient for $\rho(A)=\max_{B\in\Omega(A)}\rho(B)$ or $\rho(A)=\min_{B\in\Omega(A)}\rho(B)$.

\begin{theorem}
\label{equivalence}
Let $A\in\R_+^{n\times n}$ be irreducible and $0<x\in\R_+^n$ be a Perron eigenvector of $A$. Then
\begin{equation}
\label{max}
 \rho(A)=\max_{B\in\Omega(A)}\rho(B) \Longleftrightarrow\ \left(\forall\ i,j,k\ 1\leq i,j,k\leq n: x_k < x_j\;\Rightarrow\; a_{i,k}\leq a_{i,j}\right)
\end{equation}
\begin{equation}
\label{min}
 \rho(A)=\min_{B\in\Omega(A)}\rho(B) \Longleftrightarrow\ \left(\forall\ i,j,k\ 1\leq i,j,k\leq n: x_k < x_j\;\Rightarrow\;a_{i,k}\geq a_{i,j}\right)
\end{equation}
\end{theorem}
\begin{proof}
By Lemma \ref{primary tool}, the conditions on the right hand sides of~\eqref{max} and\eqref{min} are necessary. The fact that they they are sufficient follows immediately from Theorem \ref{sufficient}.
\end{proof}

\if{
The proof establishing equivalence  \ref{min} is similar to the proof used to establish equivalence \ref{max}. By lemma \ref{primary tool} $\rho(A)=\min_{B\in\Omega(A)}\rho(B)$ implies $\forall\ i,j,k\ 1\leq i,j,k\leq n: x_k < x_j \text{ implies }a_{i,k}\geq a_{i,j}$. 
\\
\\
To establish the converse assume   $\forall\ i,j,k\ 1\leq i,j,k\leq n$ $x_k < x_j \text{implies }a_{i,k}\geq a_{i,j}$ and $\exists\ C\in\Omega(A)$ such that $\rho(A)\geq\rho(C)=\min_{B\in\Omega(A)}\rho(B)$. 
\\
\\
 Then $\exists y\in\R_+^n$ satisfying $Cy=\rho(C)y$. Hence $\forall\ i,j,k\ 1\leq i,j,k\leq n\ y_k < y_j \text{ implies }c_{i,k}\geq c_{i,j}$. Let $\beta$ be a permutation so that $\forall i,k\ i<k \text{ implies }x_{\beta(i)}\geq x_{\beta(k)} $.  Then $\forall i$ $\beta$ acting on the $i^{th}$ row of $A$ yields a row for which the components are non increasing as their index increases. Let $\gamma$ be a permutation so that $\forall i,k\ i\leq k \text{ implies }y_{\gamma(i)}\geq y_{\gamma(k)}$. Then $\forall i$ $\gamma$ acting on the $i^{th}$ row of $C$ yields a row for which the components are non increasing as their index increases. Since $C\in\Omega(A)$ it follows that for all $i$ the non increasing row formed by  action of $\beta$ on the $i^{th}$ row of $A$ is identical to the non increasing row formed by the action of $\gamma$ on the $i^{th}$ row of $C$.

Let $P_{\beta}$ be the permutation matrix associated with $\beta$ (i.e. $P_{\beta}=(p_{i,j})$ where $p_{i,j} = 0$ for $j\neq i$ and $1$ if $j=i$). Let $P_{\gamma}$ be the permutation matrix associated with $\gamma$. Then  $AP_{\beta}=CP_{\gamma}$. Thus $AP_{\beta}P_{\gamma}^{-1}=C.$ Then $\rho(A)x=Ax\leq AP_{\beta}P_{\gamma}^{-1}x=Cx$ it follows from theorem \ref{BP28} that $\rho(C)\geq \rho(A).$ Thus $\rho(C)=\rho(A)$ and hence $\rho(A)=\min_{B\in\Omega(A)}\rho(B).$ Thus equation \ref{min} is established.
 \end{proof}
 }\fi
 
\section{Solving $\max_{B\in\Omega(A)}\rho(B)$ and $\min_{B\in\Omega(A)}\rho(B)$}

Below we give two simple iterative procedures for solving $\max_{B\in\Omega(A)}\rho(B)$ and $\min_{B\in\Omega(A)}\rho(B)$. Note that the computation of the minimum and maximum spectral radius over sets more general than $\Omega(A)$ was investigated by Protasov~\cite{P} where similar iterative procedures were suggested. 

Before presenting  the iterative procedures we first establish the following lemmas. 

\begin{lemma}
\label{PAAP}
Let $A\in\R_+^{n\times n}$ and $P$ be a permutation matrix. Then $\rho(PA)=\rho(AP)$.
\end{lemma}
\begin{proof}
It is easy to see that any eigenvalue of $PA$ is an eigenvalue of $AP$ and the other way around:
\begin{equation*}
\begin{split}
PAx=\alpha x &\Rightarrow  AP(P^{-1}x)=\alpha (P^{-1}x),\\
APy=\beta y & \Rightarrow PA(Py)=\beta (Py).
\end{split}
\end{equation*}
\end{proof}

\begin{lemma}
\label{permutation}
For $A\in\R_+^{n\times n} \text{ and all permutation matrices } P$ 
$$\max_{B\in\Omega(A)}\rho(B)= \max_{B\in\Omega(PA)}\rho(B)$$ and 
$$\min_{B\in\Omega(A)}\rho(B)= \min_{B\in\Omega(PA)}\rho(B).$$
\end{lemma}
\begin{proof}
Take arbitrary $B\in\Omega(PA)$. Then $B=PC$, where $C\in\Omega(A)$, and by Lemma~\ref{PAAP}
$\rho(B)=\rho(CP)$, where $CP\in\Omega(A)$. This observation implies that
$$
\max_{B\in\Omega(A)}\rho(B)\geq  \max_{B\in\Omega(PA)}\rho(B),\quad 
\min_{B\in\Omega(A)}\rho(B)\leq \min_{B\in\Omega(PA)}\rho(B).
$$
The reverse inequalities follow from a similar argument where we start with $B=\Omega(A)$ and 
represent $B=P^{-1}C$ with $C\in\Omega(PA)$. \end{proof}
\\
\begin{definition}
A $n\times n$ matrix $A$ is said to be fully indecomposable if $PAQ$ is irreducible for all permutation matrices $P$ and $Q$.
\end{definition}


\begin{algorithm}[h!]
\renewcommand{\algorithmicrequire}{\textbf{Input:}}
  \renewcommand{\algorithmicensure}{\textbf{Output:}}
  \caption{Solving $\max_{B\in\Omega(A)} \rho(B)$ \label{a:max}}
  \begin{algorithmic}[1]
\REQUIRE $A\in\R_+^{n\times n}$ with $A$ fully indecomposable.
\STATE Define matrix $C_0\in\Omega(A)$ by placing the entries in each row of $A$ in ascending order.  
\STATE Find a permutation matrix $Q\in\R_+^{n\times n}$ such that the Euclidean norms of the rows of $QC$ are in ascending order. 
\STATE $P\in\R_+^{n\times n}$ is the zero matrix,  $C:=QC_0$.
\WHILE{$P$ is not the identity matrix}
\STATE Find a Perron eigenvector $x\in\R_+^n$ of $C$
\IF  {the entries of $x$ are not in ascending order} 
\STATE Find a permutation matrix $P\in\R_+^{n\times n}$ so that entries of $Px$ 
are in ascending order. 
\ELSE 
\STATE Set $P$ to be the identity matrix.
\ENDIF
\STATE $C:=PC$, $Q:=PQ$.
\ENDWHILE
\ENSURE $C_0Q,$ $\rho(C_0Q)=\max_{B\in\Omega(A)} \rho(B)$
\end{algorithmic}
\end{algorithm}

\begin{algorithm}[h!]
\renewcommand{\algorithmicrequire}{\textbf{Input:}}
  \renewcommand{\algorithmicensure}{\textbf{Output:}}
  \caption{Solving $\min_{B\in\Omega(A)} \rho(B)$ \label{a:min}}
  \begin{algorithmic}[1]
\REQUIRE $A\in\R_+^{n\times n}$ with $A$ fully indecomposable.
\STATE Define matrix $C_0\in\Omega(A)$ by placing the entries in each row of $A$ in descending order.  
\STATE Find a permutation matrix $Q\in\R_+^{n\times n}$ such that the Euclidean norms of the row sums of $QC$ are in descending order. 
\STATE $P\in\R_+^{n\times n}$ is the zero matrix, $C:=QC_0$.
\WHILE{$P$ is not the identity matrix}
\STATE Find a Perron eigenvector $x\in\R_+^n$ of $C$
\IF  {the entries of $x$ are not in descending order} 
\STATE Find a permutation matrix $P\in\R_+^{n\times n}$ so that entries of $Px$ 
are in descending order. 
\ELSE 
\STATE Set $P$ to be the identity matrix.
\ENDIF
\STATE $C:=PC$, $Q:=PQ$.
\ENDWHILE
\ENSURE $C_0Q$, $\rho(C_0Q)=\min_{B\in\Omega(A)} \rho(B)$
\end{algorithmic}
\end{algorithm}

We now argue that Algorithm~\ref{a:max} is valid.
Observe that if in step 6 vector $x$ is not in ascending order and hence $P$ is not the identity matrix, 
then $C(Px)\geq Cx=\rho(C)x$ with at least one strict inequality, since all rows of $C$ as well as $Px$ 
are aligned together in ascending order, but this is not true about all rows of $C$ and vector $x$.  Then we obtain 
$(PC)Px\geq \rho(C)Px$ with at least one strict inequality, and  
by Proposition \ref{BP28} part (iii) $\rho(C)<\rho(PC)$. If $P$ is the identity matrix then $\forall\ i,j,k\ 1\leq i,j,k\leq n: x_k < x_j\;\Rightarrow\; c_{i,k}\leq c_{i,j}$  and by Theorem~\ref{equivalence} $\rho(C)=\max_{B\in\Omega(C)}\rho(B).$ By Lemma \ref{permutation} it follows that $\rho(C)=\max_{B\in\Omega(A)} \rho(B)$. The algorithm terminates in a finite 
number of iterations since $\rho(C)$ is strictly increasing so matrices $C$ do not repeat, and since the number of permutations is finite. Lemma~\ref{PAAP} also implies that for the final matrix $C$ we have $\rho(C)=\rho(C_0Q)$, implying that $C_0Q$ solves the problem of maximizing $\rho(B)$ over $\Omega(A)$ (while belonging to $\Omega(A)$).

\if{
\begin{algorithm}[h]
\renewcommand{\algorithmicrequire}{\textbf{Input:}}
 \renewcommand{\algorithmicensure}{\textbf{Output:}}
 \caption{Solving $\max_{B\in\Omega(A)} \rho(B)$ \label{a:max}}
 \begin{algorithmic}[1] 
\REQUIRE $A\in\R_+^{n\times n}$ with $A>0$.
\STATE Set $A^{(1)}:=A$ and $s:=1$
\WHILE{$1$}
\STATE Find a Perron eigenvector $x\in\R_+^n$ of $A^{(s)}$
\STATE Set $A^{(s+1)}$ as a matrix in $\Omega(A)$ s.t. 
$x_k<x_j\Rightarrow a_{i,k}^{(s+1)}\leq a_{i,j}^{(s+1)}$ $\forall i,k,j$
\IF{$\rho(A^{(s+1)})=\rho(A^{(s)})$}
\STATE \textbf{break}
\ELSE
\STATE $s:=s+1$
\ENDIF
\ENDWHILE
\ENSURE $A^{(s)}$ s.t. $\rho(A^{(s)})=\max_{B\in\Omega(A)} \rho(B)$
\end{algorithmic}
\end{algorithm}

\begin{algorithm}[h]
\renewcommand{\algorithmicrequire}{\textbf{Input:}}
 \renewcommand{\algorithmicensure}{\textbf{Output:}}
 \caption{Solving $\min_{B\in\Omega(A)} \rho(B)$ \label{a:min}}
 \begin{algorithmic}[1] 
\REQUIRE $A\in\R_+^{n\times n}$ with $A>0$.
\STATE Set $A^{(1)}:=A$ and $s:=1$
\WHILE{$1$}
\STATE Find a Perron eigenvector $x\in\R_+^n$ of $A^{(s)}$
\STATE Set $A^{(s+1)}$ as a matrix in $\Omega(A)$ s.t. 
$x_k<x_j\Rightarrow a_{i,k}^{(s+1)}\geq a_{i,j}^{(s+1)}$ $\forall i,k,j$
\IF{$\rho(A^{(s+1)})=\rho(A^{(s)})$}
\STATE \textbf{break}
\ELSE
\STATE $s:=s+1$
\ENDIF
\ENDWHILE
\ENSURE $A^{(s)}$ s.t. $\rho(A^{(s)})=\min_{B\in\Omega(A)} \rho(B)$
\end{algorithmic}
\end{algorithm}
}\fi

\if{
For Algorithm \ref{a:max}, in general, we have $\rho(A^{(s)})\leq\rho(A^{(s+1)})$ since $\rho(A^{(s})x=A^{(s)}x\leq A^{(s+1)}x$ by the rearrangement inequality.
The condition $\rho(A^{(s+1)})=\rho(A^{(s)})$ is equivalent to the property  $x_k<x_j\Rightarrow a_{i,k}^{(s)}\leq a_{i,j}^{(s)}$
$\forall i,k,j$ (and hence to $\rho(A^{(s)})=\max_{B\in\Omega(A)} \rho(B)$ by Theorem~\ref{equivalence}). Indeed, if this condition does not hold then  $\rho(A^{(s)})\neq\max_{B\in\Omega(A)} \rho(B)$, so this property does not hold by theorem~\ref{equivalence}. If the property does not hold, then $\rho(A^{(s)})x=A^{(s)}x\leq A^{(s+1)}x$ with a strict inequality in at least one component, thus $\rho(A^{(s)})<\rho(A^{(s+1)})$ by Proposition \ref{BP28} part (iii). Thus $\rho(A^{(s+1)})=\rho(A^{(s)})$ is a valid stopping condition ensuring the maximality of  $\rho(A^{(s)})$. The algorithm converges in a finite number of iterations since the Perron roots strictly increase until the optimality so there is no repetition of $A^{(s)}$, and since the set $\Omega(A)$ is finite.
}\fi

Algorithm~\ref{a:min} is valid for the reasons similar to those explained above for Algorithm~\ref{a:max}.
We now demonstrate the work of Algorithm~\ref{a:max} on the following small example.

\begin{example}
Consider matrix 
$$
A=
\begin{pmatrix}
2 & 5 & 2 & 2 & 5\\
6 & 6 & 2 & 3 & 1\\
7 & 3 & 5 & 5 & 3\\
3 & 3 & 4 & 6 & 8\\
2 & 4 & 2 & 5 & 5
\end{pmatrix}
$$
First we align all rows of this matrix in ascending order thus obtaining $C_0$. The Euclidian norms of the row sums if $C$ are $11,$ $12,$ $18,$ $27$ and $31$. Thus initially $Q=I$ and $C=QC_0=C_0$ with its Perron vector $x$: 
$$
C=QC_0=
\begin{pmatrix}
2 & 2 & 2 & 5 & 5\\
1 & 2 & 3 & 6 & 6\\
3 & 3 & 5 & 5 & 7\\
3 & 3 & 4 & 6 & 8\\
2 & 2 & 4 & 5 & 5
\end{pmatrix}, 
\quad 
x\approx 
\begin{pmatrix}
0.3561\\ 
0.4098\\ 
0.5091\\ 
0.5301\\ 
0.4063
\end{pmatrix}
$$
The components of $x$ are not ascending and we have:  
\begin{equation*}
P=
\begin{pmatrix}
1 & 0 & 0 & 0 & 0\\
0 & 0 & 0 & 0 & 1\\
0 & 1 & 0 & 0 & 0\\
0 & 0 & 1 & 0 & 0\\
0 & 0 & 0 & 1 & 0
\end{pmatrix}
\end{equation*}
The next while loop proceeds, since $P\neq I$. 
We compute the next matrix $C$ 
and its Perron eigenvector $x$:
\begin{equation*}
\begin{split}
C:=PQC_0=
\begin{pmatrix}
2 & 2 & 2 & 5 & 5\\
2 & 2 & 4 & 5 & 5\\
1 & 2 & 3 & 6 & 6\\
3 & 3 & 5 & 5 & 7\\
3 & 3 & 4 & 6 & 8
\end{pmatrix},\quad
x\approx 
\begin{pmatrix}
0.3595\\ 
0.3987\\ 
0.4116\\ 
0.5055\\ 
0.5355
\end{pmatrix}
\end{split}
\end{equation*}
Here, $x$ is in the ascending order. The algorithm ends and returns
\begin{equation*}
C_0PQ=
\begin{pmatrix}
2 & 2 & 5 & 5 & 2\\
1 & 3 & 6 & 6 & 2\\
3 & 5 & 5 & 7 & 3\\
3 & 4 & 6 & 8 & 3\\
2 & 4 & 5 & 5 & 2
\end{pmatrix}, \quad \rho(C_0PQ)\approx 20.9863.
\end{equation*}
\end{example}

\begin{remark}
\label{r:convergence}
We conducted a number of numerical experiments, in which we increased the matrix dimension from 5 to 200. For each dimension we generated 50 random instances of $A$ and counted the number of while loops that Algorithms~\ref{a:max} and~\ref{a:min} require before convergence. For the whole dimension range, the average number of while loops stayed  with the maximum number of loops not exceeding 3. Finding  a reasonable upper bound on the number of loops before convergence is an open problem. 
Note that Cvetkovi\'{c} and Protasov~\cite{ACVP} establish that a similar algorithm has local quadratic convergence (see \cite{ACVP}, page 19). 
\end{remark}
\if{
The following argument might help to explain the observations in Remark~\ref{r:convergence}.
For any $i$ and $s$ denote by $\theta_i^{(s)}$ be the angle between the Perron eigenvector $x^{(s)}$ in the $s$th while loop and the $i$th row of $C$, denoted by $C^{(s)}[i;1\ldots n]$. Assume that $\left\|x^{(s)}\right\|=1$ for any $s$. Then we can write 
$$\left(C^{(s)}x^{(s)}\right)_i=\left\|C^{(s)}[i;1\ldots n\right\|\|x^{(s)}\|cos(\theta_i^{(s)})=\rho(C^{(s)})x^{(s)}_i.$$ 
 To show that $\forall\ i\ $
 $\left\|C^{(s)}[i;1\ldots n\right\|\|x^{(s)}\|cos(\theta_i^{(s)})<\left\|C^{(s+1)}[i;1\ldots n\right\|\|x^{(s+1)}\|cos(\theta_i^{(s+1)})$ we consider the case that $x^{(s)}_i>x^{(s+1)}_i$ and the case that $x^{(s+1)}_i\geq x^{(s)}_i$. If  $x^{(s)}_i>x^{(s+1)}_i$ then $\frac{C^{(s)}x^{(s)}|_i}{\rho(C^{(s)})}=x^{(s)}_i<x^{(s+1)}=\frac{C^{(s+1)}y|_i}{\rho(C^{s+1})}<\frac{C^{(s+1)}y|_i}{\rho(C^{(s)})}$. Thus $C^{(s)}x|_i<C^{(s+1)}x^{(s+1)}|_i.$ If $x^{(s+1)}\geq x^{(s)}$ then $\frac{C^{(s)}x^{(s)}|_i}{x^{(s)}_i}=\rho(C^{(s)})<\rho(C^{(s+1)})=\frac{C^{(s+1)}x^{(s+1}|_i}{x^{(s+1)}_i}\leq\frac{C^{(s+1)}x^{(s+1}|_i}{x^{(s)}_i}.$ Thus $C^{(s)}x^{(s)}|_i<C^{(s+1)}x^{(s+1)}|_i$. Since $\forall\ i\ $
 $\left\|C^{(s)}[i;1\ldots n\right\|\|x^{(s)}\|=\left\|C^{(s+1)}[i;1\ldots n\right\|\|x^{(s+1)}\|$ it follows that in each iteration 
 the angle between the eigenvector and each matrix row vector must all be strictly decreasing. Since as the  dimension of the matrix increases the number of row vectors increases  our observation that the number of loops stays low over the whole dimension range should be expected. Comming up with a method to quantify this to yield a reasonable upper bound on the number of loops required is an open problem.

}\fi

\if{

Since $\frac{\left\|C[i;1\ldots n]\right\|\cos(\theta_i^{(s)})}{\rho(C)}=x^{(s)}_i$ we have that $\sum_{i=1}^n \frac{\left\|C[i;1\ldots n]\right\|^2\cos(\theta_i^{(s)})^2}{\rho(C)^2}=\sum_{i=1}^n (x_i^{(s)})^2=1.$ Thus each iteration as $\rho(C)$ increases the angles $\theta_i^{(s)}$ must in some sense over all be decreasing.  

When computing the maximum Perron root Algorithm \ref{a:max} in step 7 selects $\phi$ so that 
$\forall i \left\|C[\phi(i);1\ldots n]\right\|\frac{\cos(\theta_{\phi(i)}^{(s)})}{x^{(s)}_i}\geq \left\|C[i;1\ldots n]\right\|\frac{\cos(\theta_i^{(s)})}{x^{(s)}_i}= \rho(C)$ such that the number of indecencies with strict inequality are maximized. Algorithm \ref{a:max} terminates when no index has strict inequlity. Since the angles $\theta_{\phi(i)}^{(s)}$ are for the most part decreasing with each iteration the number of strict inequalities also tends to decrease. 

It also follows from $\frac{\left\|C[i;1\ldots n]\right\|\cos(\theta_i^{(s)})}{\rho(C)}=x^{(s)}_i$ that the order of the entries in 
$\left\{x^{(s)}_i:1\leq i\leq n\right\}$ is the same as the ordering of$\left\{\left\|C[i;1\ldots n]{\cos(\theta_i^{(s)})}\right\|:1\leq i\leq n\right\}$.
Since $$n\rho(C)=\sum_{i=1}^n \frac{\cos(\theta^{(s)}_i)}{x_i}\left\|C[i;1\ldots n]\right\|.$$  
it follows that
$$\sum_{i=1}^n \left\|C[i;1\ldots n]\right\|\frac{\cos(\theta^{(s)}_{i})}{x^{(s)}_{n+1-i}}\geq\sum_{i=1}^n \left\|C[\phi(i);1\ldots n]\right\|\frac{\cos(\theta^{(s)}_{\phi(i)})}{x^{(s)}_i}\geq n\rho(C).$$  
\newline
\newline
Deriving a reasonable upper bound on the number of while loops before convergence is an open problem.

Use Theorem \ref{main} to show that $\sum_{i=1}^n \left\|C[i;1\ldots n]\right\|\frac{\cos(\theta_i^{(1)})}{x^{(1)}_{n+1-i}}\geq n\max_{B\in\Omega(A)}\rho(B)$ !!

In our example we have after the first execution of the while loop:
$$\left\|C[1\ldots 5;1\ldots 5]\right\|= 7.87401\ 9.27362\ 10.8167\ 11.5758\ 8.60233$$
$$\cos(\theta_1)=0.918461\ \cos(\theta_2)=0.89731\ \cos(\theta_3)=0.955718\ \cos(\theta_4)=0.929953\ \cos(\theta_5)=0.959057$$
$$ x=0.356137\ 0.409781\ 0.509077\ 0.53012\ 0.406276$$
$\sum_{i=1}^n \left\|C[i;1\ldots n]\right\|\frac{\cos(\theta_i^{(1)})}{x^{(1)}_{n+1-i}}= 105.827$ and 
$\sum_{i=1}^n \frac{\cos(\theta^{(1)}_i)}{x^{(1)}_i}\left\|C[i;1\ldots n]\right\|=101.534=5*\rho(C).$

then after the first iteration we  have that $$21.164\geq \max_{B\in\Omega(A)}\rho(B)=20.9863  \geq 20.3067.$$ 
}\fi

\if{
For $0<A\in\R_+^{n\times n}$ the two procedures below begin by setting $A^{(1)}$ to $A$ and $s$ to $1$. 
\begin{center}
\textbf{Procedure for computing $\max_{B\in\Omega(A)}\rho(B)$}
\end{center}
\begin{enumerate}
  \item Find  $x\in\R_+^n$ satisfying $A^{(s)}x=\rho(A^{(s)})x.$ 
	\item Set $A^{(s+1)}$ to be the matrix in $\Omega(A)$ so $\forall\ i,j,k\ 1\leq i,j,k\leq n: x_k < x_j \text{ implies }a^{(s+1)}_{i,k}\leq a^{(s+1)}_{i,j}$. Then $A^{(s+1)}x\geq A^{(s)}x=\rho(A)x$ thus  by theorem \ref{BP28} $\rho(A^{(s+1)})\geq \rho( A^{(s)}).$
	\item While  $\rho(A^{(s+1)})>\rho( A^{(s)})$ increment $s$ by $1$ and go to step $1$. 
	\item By theorem \ref{equivalence} equation  \ref{max} $\rho(A^{(s+1)})=\max_{B\in\Omega(A)}\rho(B).$
\end{enumerate} 

\begin{center}
\textbf{Procedure for computing $\min_{B\in\Omega(A)}\rho(B)$}
\end{center}
\begin{enumerate}
\item Find  $x\in\R_+^n$ satisfying $A^{(s)}x=\rho(A^{(s)})x.$ 
	\item Set $A^{(s+1)}$ to be the matrix in $\Omega(A)$ so $\forall\ i,j,k\ 1\leq i,j,k\leq n: x_k < x_j \text{ implies }a^{(s+1)}_{i,k}\geq a^{(s+1)}_{i,j}$. Then $A^{(s+1)}x\leq A^{(s)}x=\rho(A)x$ thus  by theorem \ref{BP28} $\rho(A^{(s+1)})\leq \rho( A^{(s)}).$
	\item While  $\rho(A^{(s+1)})<\rho( A^{(s)})$ increment $s$ by $1$ and go to step $1$. 
	\item By theorem \ref{equivalence} equation \ref{min} $\rho(A^{(s+1)})=\min_{B\in\Omega(A)}\rho(B).$
\end{enumerate} 
}\fi

\if{
\begin{verbatim}
 A   
2 5 2 2 5
6 6 2 3 1
7 3 5 5 3
3 3 4 6 8
2 4 2 5 5
   C
2 2 2 5 5
1 2 3 6 6
3 3 5 5 7
3 3 4 6 8
2 2 4 5 5
    
 max_eig C   
eigenvalue= 20.3067  eigenvector = 0.356137 0.409781 0.509077 0.53012 0.406276

|| C||_i= 7.87401 9.27362 10.8167 11.5758 8.60233
theta_i= 0.918461 0.89731 0.955718 0.929953 0.959057
x_i/theta_i=0.387754 0.456677 0.532664 0.57005 0.42362
\sum || C||i theta_i/x_{n+1-i}=   102.725
\sum || C ||_i theta_i/x_i=101.534=5*\rho(C)
   /:0.356137 0.409781 0.509077 0.53012 0.406276  
Permutation when applied to the eigenvector to  it in ascending order 0 4 1 2 3
   
    C=. 0 4 1 2 3{C                     C=:PC
   C
2 2 2 5 5
2 2 4 5 5
1 2 3 6 6
3 3 5 5 7
3 3 4 6 8
   max_eig C
eigenvalue = 20.9863  0.359518 0.398744 0.411606 0.505523 0.535516
   eigenvector  =0.359518 0.398744 0.411606 0.505523 0.535516 is in ascending order
   P=I
	
	20.9863 is  maxiumum Perron root of a matrix in Omega(A).  
\end{verbatim}   
}\fi


\begin{thebibliography}{11}
\bibitem{Al} Yu.A. Al'pin, Bounds for the Perron root of a non negative matrix involving properties of its graph, {\em Math. Notes (Moscow),} \textbf{58}, 1995, p. 1121-1123.

\bibitem{ED} L. Elsner, P. van den Driessche, Bounds for the Perron root using max eigenvalues, {\em Linear Algebra Appl.} \textbf{428}, 2005, p. 2000-2005.

\bibitem{BP} A. Berman, R.J. Plemmons. Nonnegative Matrices in Mathematical Sciences,  Academic Press, New York et al.,1979.

\bibitem{ACVP} A. Cvetkovi\'{c} and V.Yu. Protasov, The Greedy Strategy for Optimizing the Perron eigenvalue, {\em Mathematical Programming} \textbf{193}, 2022, p. 1-31.

\bibitem{ESS} G.M. Engel, H. Schneider and S. Sergeev. On sets of eigenvalues of matrices with prescribed row sums and prescribed graph. {\em Linear Algebra Appl.} \textbf{455}, 2014, p. 187-209.

\if{
\bibitem{ANM} B. Carnahan, H. A. Luther and J. O. Wilkes, Applied Numerical Methods, {\em John Wiley and Sons}, New York, 1969.
}\fi
\bibitem{HLP} G.H. Hardy, J.E. Littlewood and G. Polya, Inequalities , 2nd edition,  Cambridge University Press, 1953.
\if{
\bibitem{AONM} Eugene Isaacson and Herbert Bishop Keller,  Analysis of Numerical Methods {\em John Wiley and Sons}, New York 1996.
}\fi


\bibitem{P} V.Yu. Protasov, Spectral simple method,
{\em Mathematical Programming, Ser.~A} \textbf{156}, 2016, p. 485-511.

\end{thebibliography}
\end{document}